\documentclass[11pt]{amsart}

\usepackage[pagebackref=true,colorlinks,linkcolor=blue,citecolor=magenta]{hyperref}
\usepackage{amsmath,amssymb,latexsym} 
\usepackage{graphicx}
\usepackage{fullpage}
\usepackage[square, numbers]{natbib}   
\usepackage{comment}

\newtheorem{theorem}{Theorem} 

\newtheorem{proposition}{Proposition}

\theoremstyle{definition}

\newtheorem{remark}{Remark}

\title{Frequency-Constrained Edge Coloring of Hypergraphs}
\author{Saeed Shaebani}
\address{S. Shaebani, 
	School of Mathematics and Computer Science,
	Damghan University, Damghan, Iran.}
\email{shaebani@du.ac.ir}

\author{Meysam Alishahi}
\address{M. Alishahi, 
	Faculty of Mathematical Sciences,
	Shahrood University of Technology, Shahrood, Iran.}
\email{meysam\_alishahi@shahroodut.ac.ir}

\begin{document}
\begin{abstract}
\noindent Suppose that a hypergraph ${\mathcal H}$ and an arbitrary nonempty (finite or infinite) set of available colors are given.
Each color $x$ is associated with a frequency $\tau (x)$, where the set of all such frequencies is bounded.
We define a new parameter called the 
{\it $\tau$-matching chromatic number}, denoted by $\chi_M(\tau, {\mathcal H})$,
as the least possible number of colors required to color the edges of
${\mathcal H}$ in such a way that the size of each nonempty monochromatic matching does not exceed the frequency of the corresponding color associated to its edges.
The well-known and extensively well-studied chromatic number of general Kneser hypergraph $\chi \left( {\rm KG}^r({\mathcal H}) \right)$ is a special case of $\chi_M(\tau, {\mathcal H})$ when all color frequencies are the fixed constant $r-1$.

In this paper, we establish sharp lower bounds for the parameter
$\chi_M(\tau , {\mathcal H})$, utilizing the concepts of the alternation number and the equitable colorability defect.

\end{abstract}

\maketitle
\section{{\bf Matching Coloring of Hypergraphs}}

A {\it hypergraph} ${\mathcal H}$ is an ordered pair ${\mathcal H}=(V,E)$, where $V$ is a finite set (called the {\it vertex set} of ${\mathcal H}$) and $E$ is a set of some nonempty subsets of $V$ (called the {\it edge set} of ${\mathcal H}$). A hypergraph ${\mathcal H}$ is said to be {\it $r$-uniform} if each of its edges is of size $r$, that is, whenever $|e| = r$ for each edge $e$ of ${\mathcal H}$.

By a {\it proper vertex coloring} of a (not necessarily uniform) hypergraph ${\mathcal H}$, we mean a function
$c:V({\mathcal H}) \rightarrow C$ such that for each $S\in E({\mathcal H})$ with $|S|\geq 2$ we have $|\{c(v) : v\in S\}| \geq 2$. Such a set $C$ is called the set of {\it colors}. By a {\it singleton edge} of ${\mathcal H}$, we mean an edge $e$ with $|e|=1$. Whenever ${\mathcal H}$ has no singleton edges, we define the {\it chromatic number} of ${\mathcal H}$, denoted by $\chi ({\mathcal H})$, as the minimum cardinality of a set $C$ such that a proper vertex coloring $f:V({\mathcal H})\rightarrow C$ exists. If ${\mathcal H}$ has some singleton edges, then we define $\chi ({\mathcal H}) :=+\infty$.

As a definition, any subset of $E({\mathcal{H}})$ which consists of some pairwise disjoint edges of ${\mathcal{H}}$ is called a {\it matching}; and an {\it $l$-matching} is a matching of size
$l$, that is, a set consisting of $l$ pairwise disjoint edges of ${\mathcal{H}}$.

Let ${\mathcal H}$ be a hypergraph with $n$ vertices, and $r$ be a positive integer. Also, let $C$ be an arbitrary nonempty (finite or infinite) set, and regard its elements as our available colors.

\noindent Any arbitrary mapping $\tau: C\longrightarrow \{0,1,\ldots,r-1\}$ is called an {\it $r$-color-frequency mapping}. So, by an $r$-color-frequency mapping, each color in $C$ receives a nonnegative integer less that $r$ as its {\it frequency}. Our aim is coloring the edges of ${\mathcal H}$ by some colors in $C$ in such a way that the size of each nonempty monochromatic matching does not exceed the frequency of the color associated to its edges. Naturally, one is looking for the least possible number of colors needed for such a coloring.

\noindent Formally, for a subset $A$ of $C$, an {\it $(A,\tau)$-matching coloring} of ${\mathcal H}$ is
a coloring of edges of ${\mathcal H}$ by some colors in $A$ in such a way that the size of each nonempty monochromatic matching is bounded above by the frequency of the corresponding color associated to its edges.
The hypergraph ${\mathcal H}$ is called {\it $\tau$-admissible} if ${\mathcal H}$ admits an $(A,\tau)$-matching coloring for some finite subset $A$ of $C$; and in this case, we define the {\it $\tau$-matching chromatic number} of ${\mathcal H}$, denoted by $\chi_M(\tau, {\mathcal H})$, equal to the minimum cardinality of such a finite subset $A$ of $C$. As usual, for the case where ${\mathcal H}$ is not $\tau$-admissible, we set $\chi_M(\tau, {\mathcal H}):=+\infty$.

The {\it general Kneser hypergraph} ${\rm KG}^r({\mathcal H})$ is an $r$-uniform hypergraph which has $E({\mathcal H})$ as its vertex set and  whose edges are formed by $r$ pairwise disjoint edges of ${\mathcal H}$, that is,
$$E\left({\rm KG}^r({\mathcal H})\right) :=\left\{\{e_1,\ldots,e_r\}:\; e_i\cap e_j=\varnothing \mbox{ for all } i\neq j\right\}.$$
It is worth pointing out that the well-known and extensively well-studied chromatic number of general Kneser hypergraph $\chi \left( {\rm KG}^r({\mathcal H}) \right)$ is a special case of $\chi_M(\tau, {\mathcal H})$ where all color frequencies are the fixed constant $r-1$.

In the rest of this paper, we aim to establish sharp lower bounds for the parameter
$\chi_M(\tau , {\mathcal H})$, utilizing the concepts of the alternation number and the equitable colorability defect.

\maketitle
\section{{\bf Alternation Number of Hypergraphs}}

This section is devoted to provide a sharp lower bound for $\chi_M(\tau, {\mathcal H})$ in terms of alternation number.

For any nonnegative integer $n$, let the symbol $[n]$ denote the set $\{1, 2, \dots , n\}$.
Also, for a positive integer $r\geq 2$,
let $\mathbb{Z}_r=\{\omega,\omega^2,\ldots,\omega^r\}$ be a multiplicative cyclic group of order $r$  with a generator $\omega$ in such a way that
$0 \notin \{\omega,\omega^2,\ldots,\omega^r\}$.

\noindent For a vector $X=(x_1,x_2,\ldots,x_n)\in(\mathbb{Z}_r\cup\{0\})^n\setminus\{(0,0, \dots , 0)\}$, a subsequence of nonzero terms of $X$ is called {\it alternating} if each two consecutive terms of this subsequence are different.  The length of the longest alternating subsequence of $X$ is denoted by ${\rm alt}(X)$. For example,
if we set $r=3$, $n=14$, and $X=(\omega^2 , 0 , 0 , \omega^2,\omega , 0 ,\omega^3 , 0 , \omega^2 , \omega^3 , 0 , \omega^3 , \omega , 0)$, then ${\rm alt}(X)=6$.
Also, we define ${\rm alt}(0,0,\dots,0)$ to be zero.
For an $X=(x_1,x_2,\ldots,x_n)\in(\mathbb{Z}_r\cup\{0\})^n$ and $\epsilon\in \mathbb{Z}_r$, we define $X^\epsilon\subseteq [n]$ to be the set of all indices $j$ such that $x_j=\epsilon$, that is,
$X^\epsilon :=\{j:\; x_j=\epsilon\}$.
By abuse of language, we can write $X=\displaystyle{\left(X^\epsilon\right)}_{\epsilon\in\mathbb{Z}_r}$.

Let ${\mathcal H}$ be a hypergraph with $n$ vertices and $\sigma:[n]\longrightarrow V({\mathcal H})$ be an injective mapping.
We define ${\rm alt}_r({\mathcal H},\sigma)$ to be the maximum possible value of
${\rm alt} (X)$ for  $X=\displaystyle{\left(X^\epsilon\right)}_{\epsilon\in\mathbb{Z}_r}\in (\mathbb{Z}_r\cup\{0\})^n$ such that none of $\sigma(X^\epsilon)$ contains any edge of ${\mathcal H}$.
In other words,
$${\rm alt}_r({\mathcal H},\sigma) :=\max\left\{{\rm alt}(X): X\in(\mathbb{Z}_r\cup\{0\})^n \mbox{ and for each } \epsilon\in \mathbb{Z}_r, \mbox{ we have } E({\mathcal H}[\sigma\left(X^\epsilon\right)])=\varnothing  \right\}.$$
The {\it alternation number of ${\mathcal H}$}, denoted by ${\rm alt}_r({\mathcal H})$, is the minimum possible value for ${\rm alt}_r({\mathcal H},\sigma)$ where the minimum is taken over all
injective mappings $\sigma:[n]\longrightarrow V({\mathcal H})$.

Alishahi and Hajiabolhassan ~\cite{2013arXiv1302.5394A} presented a lower bound for the chromatic number of general Kneser hypergraphs ${\rm KG}^r({\mathcal H})$ in terms of $n$ and ${\rm alt}_r({\mathcal H})$.
\begin{theorem}{\rm (\cite{2013arXiv1302.5394A})}\label{alihaji}
	For an integer $r\geq 2$ and a hypergraph ${\mathcal H}$, we have
	$$\chi\left({\rm KG}^r({\mathcal H})\right)\geq \left\lceil{ |V({\mathcal H})|-{\rm alt}_r({\mathcal H})\over r-1}\right\rceil.$$
\end{theorem}
The notion of alternation number and in particular, the Theorem~\ref{alihaji} were applied on various families of hypergraphs to compute their chromatic numbers or to investigate their coloring properties;
see~\cite{arsmc/SaniAT18,journals/combinatorics/Alishahi17,2013arXiv1302.5394A,2014arXiv1401.0138A,2014arXiv1403.4404A,jgt.22215,2014arXiv1407.8035A,Alishahi_Hossein_2020}.

In the rest of this section, a development of Theorem \ref{alihaji} is established. In this regard, the following theorem provides a lower bound for the $\tau$-matching chromatic number of hypergraphs.

\begin{theorem}\label{mainthmSaeed}
	Let $r\geq 2$ be an integer, $C$ be a nonempty finite or infinite set, and
	$\tau: C \longrightarrow \{0,1,\ldots,r-1\}$ be an $r$-color-frequency  mapping.
	Then, for any $\tau$-admissible hypergraph ${\mathcal H}$ there exists a finite subset
	$B$ of $C$ for which
	$$\sum_{b\in B}\tau(b) \geq  |V({\mathcal H})|-{\rm alt}_r({\mathcal H}) .$$
	Besides, we have
	$$\chi_M(\tau,{\mathcal H})\geq \min \left\{|A|:\;   A \subseteq C, \mbox{ and } A \mbox{ is finite,}  \mbox{ and also } \sum_{a\in A}\tau(a) \geq  |V({\mathcal H})|-{\rm alt}_r({\mathcal H})\right\}.$$
\end{theorem}

\begin{remark}
	The Theorem~\ref{mainthmSaeed} can be viewed as a generalization of Theorem~\ref{alihaji}. To illustrate this, let us consider the set of available colors $C$ to be the set of natural numbers
	$\mathbb{N}$, and define a constant $r$-color-frequency mapping $\tau: \mathbb{N} \longrightarrow \{0,1,\ldots,r-1\}$ by setting $\tau(a) = r-1$ for every $a \in \mathbb{N}$. Under this construction, one obtains:
	$$\chi_M(\tau, {\mathcal H})=\chi\left({\rm KG}^r({\mathcal H})\right)$$
	and also,
	$$\min\left\{|A|:\;   A \subseteq C, \mbox{ and } A \mbox{ is finite,}  \mbox{ and } \sum_{a\in A}\tau(a) \geq  |V({\mathcal H})|-{\rm alt}_r({\mathcal H})\right\}=\left\lceil{ |V({\mathcal H})|-{\rm alt}_r({\mathcal H})\over r-1}\right\rceil.$$
	Therefore, the assertion of Theorem \ref{mainthmSaeed} implies 
	$\chi\left({\rm KG}^r({\mathcal H})\right)\geq \left\lceil{ |V({\mathcal H})|-{\rm alt}_r({\mathcal H})\over r-1}\right\rceil.$
\end{remark}

\vspace{0.4cm}

Having concluded the above remark, we now return to the main argument and present the proof of Theorem \ref{mainthmSaeed}, as follows.

\noindent{\bf Proof of Theorem~\ref{mainthmSaeed}.}
Let us regard an arbitrary finite subset $A$ of $C$ for which ${\mathcal H}$ admits an $(A,\tau)$-matching coloring $c: E({\mathcal H})\longrightarrow A$.
Now, for each $a$ in $A$ with $\tau(a) < r-1$, we add $r-1-\tau(a)$ additional vertices
$x^{(a)}_{1},x^{(a)}_{2},\ldots ,x^{(a)}_{r-1-\tau(a)}$ to ${\mathcal H}$
together with adding all additional singleton edges
$\left\{x^{(a)}_{1}\right\},\left\{x^{(a)}_{2}\right\},\ldots ,\left\{x^{(a)}_{r-1-\tau(a)}\right\}$ to $E({\mathcal H})$
in order to obtain a new hypergraph $\widehat{{\mathcal H}}$.
Also, we extend the edge coloring $c: E({\mathcal H})\longrightarrow A$
to an edge coloring of $\widehat{{\mathcal H}}$, say $\widehat{c}: E\left(\widehat{{\mathcal H}}\right)\longrightarrow A$, in such a way that for each $a$ in $A$ with $\tau(a) < r-1$, all additional singleton edges
$\left\{x^{(a)}_{1}\right\},\left\{x^{(a)}_{2}\right\},\ldots ,\left\{x^{(a)}_{r-1-\tau(a)}\right\}$ are colored by $a$.
On one hand, the function $c: E({\mathcal H})\longrightarrow A$ is an $(A,\tau)$-matching coloring
of ${\mathcal H}$. Hence, for each $a$ in $A$, the size of every matching in ${\mathcal H}$
whose edges are colored by $a$ is less than or equal to $\tau (a)$.
On the other hand,
there are exactly $r-1-\tau(a)$ edges in $E\left(\widehat{{\mathcal H}}\right) \setminus E({\mathcal H})$
that are colored by $a$ under $\widehat{c}: E\left(\widehat{{\mathcal H}}\right)\longrightarrow A$. We conclude that for each $a$ in $A$, the size of every matching in $\widehat{{\mathcal H}}$ whose edges are colored by $a$ is at most $r-1$. Therefore, the function
$\widehat{c}: E\left(\widehat{{\mathcal H}}\right)\longrightarrow A$
is a proper vertex coloring of ${\rm KG}^r\left(\widehat{{\mathcal H}}\right)$. So,
we have
$$|A| \geq \chi\left({\rm KG}^r\left(\widehat{{\mathcal H}}\right)\right)\geq { \left|V\left(\widehat{{\mathcal H}}\right)\right|-{\rm alt}_r\left(\widehat{{\mathcal H}}\right)\over r-1};$$
and therefore, 
$$(r-1)|A| \geq \left|V\left(\widehat{{\mathcal H}}\right)\right|-{\rm alt}_r\left(\widehat{{\mathcal H}}\right).$$
Since 
$(r-1)|A| = \left( \displaystyle\sum_{a\in A}\tau(a) \right) + \left|V\left(\widehat{{\mathcal H}}\right)\right|-|V({\mathcal H})|$ and ${\rm alt}_r\left(\widehat{{\mathcal H}}\right)={\rm alt}_r({\mathcal H})$,
we have
\begin{center}
	$\left( \displaystyle\sum_{a\in A}\tau(a) \right) + \left|V\left(\widehat{{\mathcal H}}\right)\right|-|V({\mathcal H})| \geq \left|V\left(\widehat{{\mathcal H}}\right)\right|-{\rm alt}_r({\mathcal H});$
\end{center}
and therefore,
$$\displaystyle\sum_{a\in A}\tau(a) \geq  |V({\mathcal H})|-{\rm alt}_r({\mathcal H}) .$$
We conclude that for any finite subset $A$ of $C$ for which ${\mathcal{H}}$ admits an $(A,\tau)$-matching coloring, the relation
$\displaystyle\sum_{a\in A}\tau(a) \geq  |V({\mathcal H})|-{\rm alt}_r({\mathcal H}) $ is also satisfied. Accordingly,
$$\chi_M(\tau,{\mathcal H})\geq \min \left\{|A|:\;   A \subseteq C, \mbox{ and } A \mbox{ is finite,}  \mbox{ and also } \sum_{a\in A}\tau(a) \geq  |V({\mathcal H})|-{\rm alt}_r({\mathcal H})\right\} ;$$
as desired.

\hfill$\square$

In the proof of Theorem~\ref{mainthmSaeed}, the procedure is showing that for every finite subset $A$ of $C$
that an $(A,\tau)$-matching coloring  exists, we have
$ \displaystyle\sum_{a\in A}\tau(a) \geq  |V({\mathcal H})|-{\rm alt}_r({\mathcal H}).$
If ${\mathcal H}$ has at least one edge, then every $(A,\tau)$-matching coloring of ${\mathcal H}$ satisfies $ A\neq\varnothing$. So, one can also rewrite the assertion of Theorem~\ref{mainthmSaeed}
as follows:
$$\chi_M(\tau,{\mathcal H})\geq \min \left\{|A|:\;   \varnothing \neq A \subseteq C, \mbox{ and } A \mbox{ is finite,}  \mbox{ and also } \sum_{a\in A}\tau(a) \geq  |V({\mathcal H})|-{\rm alt}_r({\mathcal H})\right\} .$$



\maketitle
\section{{\bf Equitable Colorability Defect of Hypergraphs}}

\def\cprime{$'$} \def\cprime{$'$}
For two positive integers $n$ and $k$, the symbol ${[n]\choose k}$  denotes the set of all $k$-subsets of $[n]$.
The {\it complete $k$-uniform  hypergraph}  $K_n^k$ is a hypergraph with vertex set $[n]$ and the edge set ${[n]\choose k}$.

\noindent In 1955, Kneser \cite{MR0068536} proved that $\chi\left({\rm KG}^2 \left(  K_n^k \right) \right) \leq n-2k+2$ for $n\geq 2k-1$, and he conjectured that 
$$\begin{array}{lcccr}
	\chi\left({\rm KG}^2 \left(  K_n^k \right) \right) = n-2k+2  &    &  {\rm for }  &   &   n\geq 2k-1 .
\end{array}$$
Also, Erd\H{o}s
\cite{MR0465878} in 1976 conjectured that for $r\geq 2$ we have
$$\begin{array}{lrr}
	\chi\left({\rm KG}^r \left(  K_n^k \right) \right) = \left\lceil \frac{n-r(k-1)}{r-1} \right\rceil &  {\rm if }  &  n \geq r(k-1) + 1 .
\end{array}$$
The former conjecture was shown to be correct by Lov{\'a}sz \cite{MR514625} in 1978, and the latter conjecture was also settled in 1986
by Alon, Frankl, and Lov{\'a}sz \cite{MR857448}.
As a generalization of these results, Dol{\cprime}nikov and K{\v{r}}{\'{\i}}{\v{z}} \cite{MR953021,MR1081939,MR1665335}
proved that for all hypergraphs ${\mathcal H}$, the relation
$$\chi\left({\rm KG}^r   ({\mathcal H})  \right) \geq \left\lceil \frac{ {\rm cd}^r({\mathcal H}) }{r-1} \right\rceil $$
holds, where the symbol ${\rm cd}^r({\mathcal H})$, called the {\it $r$-th colorability defect} of ${\mathcal H}$, is the least possible number of vertices that one must remove from
${\mathcal H}$ in such a way that the vertices of the remaining subhypergraph could be properly colored by at most $r$ colors. This lower bound for
$\chi\left({\rm KG}^r   ({\mathcal H})  \right) $ is sharp since for
$n \geq r(k-1) + 1$ we have
${\rm cd}^r \left(  K_n^k \right) = n-r(k-1)$; and therefore,
$$\begin{array}{lrr}
	\chi\left({\rm KG}^r \left(  K_n^k \right) \right) = \left\lceil \frac{{\rm cd}^r \left(  K_n^k \right)}{r-1} \right\rceil &  {\rm if }  &  n \geq r(k-1) + 1 .
\end{array}$$
Later, Abyazi Sani and Alishahi \cite{ABYAZISANI2018229}, and Azarpendar and Jafari \cite{AZARPENDAR2023103664},
refined Dol{\cprime}nikov and K{\v{r}}{\'{\i}}{\v{z}}'s Theorem by showing that
$$\chi\left({\rm KG}^r ({\mathcal H}) \right) \geq \left\lceil \frac{ {\rm ecd}^r({\mathcal H}) }{r-1} \right\rceil ,$$
where ${\rm ecd}^r({\mathcal H})$, referred to as the {\it $r$-th equitable colorability defect} of ${\mathcal H}$, denotes the minimum number of vertices that must be
removed from ${\mathcal H}$ in order to obtain a remaining subhypergraph that admits a proper vertex coloring using $r$ colors, such that the sizes of the color classes differ by at most one.
This inequality is also sharp because
$$\begin{array}{lrr}
	\chi\left({\rm KG}^r \left(  K_n^k \right) \right) = \left\lceil \frac{{\rm ecd}^r \left(  K_n^k \right)}{r-1} \right\rceil &  {\rm whenever }  &  n \geq r(k-1) + 1 .
\end{array}$$

One observes that $\left\lceil \frac{ {\rm ecd}^r({\mathcal H}) }{r-1} \right\rceil  \geq \left\lceil \frac{ {\rm cd}^r({\mathcal H}) }{r-1} \right\rceil $
holds. Also, it was shown in \cite{2013arXiv1302.5394A} that
$$\left\lceil{ |V({\mathcal H})|-{\rm alt}_r({\mathcal H})\over r-1}\right\rceil \geq \left\lceil \frac{ {\rm cd}^r({\mathcal H}) }{r-1} \right\rceil .$$
But it was shown in \cite{ABYAZISANI2018229,arsmc/SaniAT18} that two lower bounds
$\left\lceil{ |V({\mathcal H})|-{\rm alt}_r({\mathcal H})\over r-1}\right\rceil$ and
$\left\lceil \frac{ {\rm ecd}^r({\mathcal H}) }{r-1} \right\rceil$ are not comparable to each other in the sense that in
$$\left\lceil{ |V({\mathcal H})|-{\rm alt}_r({\mathcal H})\over r-1}\right\rceil \bigcirc \left\lceil \frac{ {\rm ecd}^r({\mathcal H}) }{r-1} \right\rceil ,$$
each of the three main comparing symbols $<$ , $=$ , and $>$ could be replaced instead of $\bigcirc$, depending on the structure of ${\mathcal{H}}$.

In this section, we are concerned with providing a tight lower bound for $\chi_M(\tau, {\mathcal H})$ in terms of equitable colorability defect.
In this regard, we state and prove the following general Theorem \ref{mainthmSaeed2}.
As a specific instance of the assertion of Theorem \ref{mainthmSaeed2}, the celebrated inequality
$\chi\left({\rm KG}^r ({\mathcal H}) \right) \geq \left\lceil \frac{ {\rm ecd}^r({\mathcal H}) }{r-1} \right\rceil $
follows by considering the fixed constant $r$-color-frequency mapping
$\tau: \mathbb{N} \longrightarrow \{0,1,\ldots,r-1\}$ with $\tau (a) = r-1$ for all elements $a$ in $ \mathbb{N} $.

\begin{theorem} \label{mainthmSaeed2}
	Let $C$ be a nonempty finite or infinite set, and $r$ be an integer with $r\geq 2$.
	Also, let the function $\tau: C \longrightarrow \{0,1,\ldots,r-1\}$ be an $r$-color-frequency mapping, and ${\mathcal H}$ be a $\tau$-admissible hypergraph. Then
	there exists a finite subset
	$B$ of $C$ for which
	$$\sum_{b\in B}\tau(b) \geq  {\rm ecd}^r({\mathcal H}) .$$
	Furthermore,
	$$\chi_M(\tau,{\mathcal H})\geq \min\left\{|A|:\   A \subseteq C, \mbox{ and } A \mbox{ is finite,}  \mbox{ and also } \sum_{a\in A}\tau(a) \geq  {\rm ecd}^r({\mathcal H})\right\}.$$
\end{theorem}

\begin{proof}{
		The main idea of the proof is to show that if ${\mathcal H}$ admits an $(A,\tau)$-matching coloring for some finite subset $A$ of $C$, then we would have $\displaystyle\sum_{a\in A}\tau(a) \geq  {\rm ecd}^r({\mathcal H})$.
		
		Let us regard an arbitrary finite subset $A$ of $C$ for which ${\mathcal H}$ admits an $(A,\tau)$-matching coloring $c: E({\mathcal H})\longrightarrow A$. We construct a new hypergraph $\widetilde{{\mathcal H}}$ in such a way that for each element $a$ of $A$ with $\tau(a) < r-1$, we add $r-1-\tau(a)$ new vertices
		$v^{(a)}_{1},v^{(a)}_{2},\ldots ,v^{(a)}_{r-1-\tau(a)}$ to the vertex set of ${\mathcal H}$ and also we add all new singleton edges $\left\{v^{(a)}_{1}\right\},\left\{v^{(a)}_{2}\right\},\ldots ,\left\{v^{(a)}_{r-1-\tau(a)}\right\}$ to the edge set of ${\mathcal H}$. Denote the resulting hypergraph by $\widetilde{{\mathcal H}}$. Now, extend the coloring $c: E({\mathcal H})\longrightarrow A$ to a mapping $\widetilde{c}: E \left(\widetilde{{\mathcal H}}\right)\longrightarrow A$ with the property that for each element $a$ in $A$ with $\tau(a) < r-1$ we have
		$$\widetilde{c}\left(\left\{v^{(a)}_{1}\right\}\right)=\widetilde{c}\left(\left\{v^{(a)}_{2}\right\}\right)=\cdots =\widetilde{c}\left(\left\{v^{(a)}_{r-1-\tau(a)}\right\}\right)=a.$$
		We claim that $\widetilde{c}$ is a proper vertex coloring of ${\rm KG}^r\left(\widetilde{{\mathcal H}}\right)$.
		Suppose the assertion of this claim is false. Then, there exist some $a$ in $A$ and an edge of ${\rm KG}^r\left(\widetilde{{\mathcal H}}\right)$, say
		$\{e_{1},e_{2},\ldots , e_{r}\}$, in such a way that $\widetilde{c}(e_{1})=\widetilde{c}(e_{2})=\cdots = \widetilde{c}(e_{r}) =a$ holds.
		Since $|\{e_{1},e_{2},\ldots , e_{r}\}\setminus E(\mathcal{H})|\leq r-1-\tau (a)$,
		we obtain the existence of $\tau (a) +1$ pairwise disjoint edges of $\mathcal{H}$ that are colored by $a$ under $c: E({\mathcal H})\longrightarrow A$, contradicting the fact that $c$ is an $(A,\tau)$-matching coloring of $\mathcal{H}$.
		
		From what has already been proved, it may be concluded that 
		$$|A| \geq \chi\left({\rm KG}^r\left(\widetilde{{\mathcal H}}\right)\right)\geq { {\rm ecd}^r\left(\widetilde{{\mathcal H}}\right)\over r-1};$$
		and consequently, it follows that 
		$$(r-1)|A| \geq {\rm ecd}^r\left(\widetilde{{\mathcal H}}\right).$$
		
		We recall that ${\rm ecd}^r\left(\widetilde{{\mathcal H}}\right)$ denotes the minimum number of vertices that must be
		removed from $\widetilde{{\mathcal H}}$ in order to obtain a remaining {\it equitably $r$-colorable} subhypergraph, that is, a hypergraph which admits a proper vertex coloring using $r$ colors, such that the sizes of the color classes differ by at most one.
		
		\noindent
		On one hand, if a hypergraph contains a singleton edge, it is never $r$-colorable. Therefore, in order to obtain an induced subhypergraph of $\widetilde{\mathcal{H}}$ that is equitably $r$-colorable, it is necessary to remove all vertices in $V(\widetilde{\mathcal{H}}) \setminus V(\mathcal{H})$. On the other hand, eliminating singleton edges alone
		may be insufficient; at least ${\rm ecd}^r(\mathcal{H})$ vertices from $\mathcal{H}$ must also be removed to acquire a remaining subhypergraph that is equitably $r$-colorable.
		Thus, 
		$${\rm ecd}^r\left(\widetilde{{\mathcal H}}\right) \geq \left|V\left(\widetilde{{\mathcal H}}\right)\right|-|V({\mathcal H})|+{\rm ecd}^r({\mathcal H}).$$
		
		Now, since 
		$(r-1)|A| = \left( \displaystyle\sum_{a\in A}\tau(a) \right) + \left|V\left(\widetilde{{\mathcal H}}\right)\right|-|V({\mathcal H})|$,
		the inequality
		$(r-1)|A| \geq {\rm ecd}^r\left(\widetilde{{\mathcal H}}\right)$
		implies
		$\displaystyle\sum_{a\in A}\tau(a) \geq  {\rm ecd}^r({\mathcal H})$;
		which is the desired conclusion.
	}
\end{proof}

Again, for the case where $E({\mathcal{H}}) \neq \varnothing$, one may rewrite the assertion of Theorem~\ref{mainthmSaeed2}
as follows:
$$\chi_M(\tau,{\mathcal H})\geq \min\left\{|A|:\   \varnothing \neq A \subseteq C, \mbox{ and } A \mbox{ is finite,}  \mbox{ and also } \sum_{a\in A}\tau(a) \geq  {\rm ecd}^r({\mathcal H})\right\}.$$


\maketitle
\section{{\bf Sharpness of The Lower Bounds}}

This section verifies the sharpness of the lower bounds given in Theorems \ref{mainthmSaeed} and \ref{mainthmSaeed2} for $\chi_M(\tau,{\mathcal H})$, and concludes the article by presenting the following proposition.

\begin{proposition}
	The lower bounds for $\chi_M(\tau,{\mathcal H})$ stated in Theorem \ref{mainthmSaeed} and Theorem \ref{mainthmSaeed2} are sharp.
\end{proposition}
\begin{proof}
	Let $n$, $k$, and $r$ be positive integers with $r\geq 2$ and
	$n\geq rk$. Also, let $\tau:\mathbb{N}\longrightarrow \{0,1,\ldots,r-1\}$
	be an $r$-color-frequency mapping such that $K_n^k$ is $\tau$-admissible and $r-1\in\tau(\mathbb{N})$.
	
	On account of Theorem \ref{mainthmSaeed}, the $\tau$-admissibility of $K_n^k$ implies
	$$\left\{|A|:\; A\subseteq \mathbb{N},\ A \mbox{ is finite, and } \sum_{a\in A}\tau(a) \geq  \left|V\left(K_n^k\right)\right|-{\rm alt}_r\left(K_n^k\right)\right\} \neq \varnothing ;$$
	and also,
	$$\chi_M\left(\tau,K_n^k\right) \geq \min \left\{|A|:\; A\subseteq \mathbb{N},\ A \mbox{ is finite, and } \sum_{a\in A}\tau(a) \geq  \left|V\left(K_n^k\right)\right|-{\rm alt}_r\left(K_n^k\right)\right\} .$$
	According to Theorem \ref{mainthmSaeed2}, the $\tau$-admissibility of $K_n^k$ indicates that
	$$\left\{|A|:\; A\subseteq \mathbb{N},\ A \mbox{ is finite, and } \sum_{a\in A}\tau(a) \geq  {\rm ecd}_r\left(K_n^k\right)\right\} \neq \varnothing ;$$
	and in addition,
	$$\chi_M\left(\tau,K_n^k\right) \geq \min \left\{|A|:\; A\subseteq \mathbb{N},\ A \mbox{ is finite, and } \sum_{a\in A}\tau(a) \geq  {\rm ecd}_r\left(K_n^k\right)\right\} .$$
	Now, since ${\rm alt}_r\left(K_n^k\right)= r(k-1)$ and ${\rm ecd}_r\left(K_n^k\right)= n - r(k-1)$,
	one finds that
	$$\left|V\left(K_n^k\right)\right|-{\rm alt}_r\left(K_n^k\right) = {\rm ecd}_r\left(K_n^k\right)= n - r(k-1) ;$$
	and therefore, we shall have established the Proposition if we show that
	$$\chi_M\left(\tau,K_n^k\right)\leq \min \left\{|A|:\; A\subseteq \mathbb{N},\ A \mbox{ is finite, and } \sum_{a\in A}\tau(a) \geq  n-r(k-1)\right\} .$$
	To do this, let us regard an arbitrary nonempty finite subset $A_0:=\{a_1,a_2,\ldots,a_t\}$ of $\mathbb{N}$ with the following four properties:
	\begin{itemize}
		\item $\displaystyle\sum_{a\in A_0}\tau(a) \geq  n-r(k-1) $; \\
		\item $\left| A_0 \right| = \min \left\{|A|:\; A\subseteq \mathbb{N},\ A \mbox{ is finite, and } \displaystyle\sum_{a\in A}\tau(a) \geq  n-r(k-1)\right\} $; \\
		\item $\tau(a_{1})\leq\tau(a_{2})\leq\cdots\leq\tau(a_{t})$; \\
		\item $\tau(a_{t})=r-1$ (Without loss of generality, the case 
		$\tau(a_{t})=r-1$
		may be considered. Otherwise, due to
		$r-1\in\tau(\mathbb{N})$, an element
		$b \in \mathbb{N}$ satisfying $\tau(b)=r-1$
		may be substituted for 
		$a_{t}$).
	\end{itemize}
	Now, let us consider $t$ pairwise disjoint sets $S_1, S_2, \ldots, S_t$ such that
	the following three conditions hold:
	\begin{itemize}
		\item $S_{1} \cup S_{2} \cup \dots \cup S_{t}=[n]$;
		\item For each $i$ in $\{1,2,\ldots ,t-1\}$ we have $|S_i|\leq \tau (a_i)$;
		\item $ |S_{t}| \leq \tau (a_t) + r(k-1) = r-1 + r(k-1)=rk-1$.
	\end{itemize}
	Also, let us define a mapping $c:\; E\left(K_n^k\right)\longrightarrow A_0$ such that each $e\in E\left(K_n^k\right)$ is mapped to $c(e):=a_{\varphi (e)}$ where
	$$ \varphi (e) := \min \left\{i:\; e\cap S_i\neq \varnothing \right\} . $$
	The mapping $c$ is an $\left(A_0 ,\tau \right)$-matching coloring of $K_n^k$; and therefore,
	$$\chi_M\left(\tau,K_n^k\right)\leq \left| A_0 \right| = \min \left\{|A|:\; A\subseteq \mathbb{N},\ A \mbox{ is finite, and } \sum_{a\in A}\tau(a) \geq  n-r(k-1)\right\} ;$$
	which is desired.
\end{proof}

\def\cprime{$'$} \def\cprime{$'$}

\bibliographystyle{plain}

\def\cprime{$'$} \def\cprime{$'$}

\end{document}